\documentclass{amsart}
\usepackage{amssymb}
\newtheorem{theorem}{Theorem}[section]
\newtheorem{corollary}[theorem]{Corollary}
\newtheorem{lemma}[theorem]{Lemma}
\newtheorem{proposition}[theorem]{Proposition}
\newtheorem{example}[theorem]{Example}
\theoremstyle{definition}
\newtheorem{definition}[theorem]{Definition}
\newtheorem{remark}[theorem]{Remark}

\numberwithin{equation}{section}

\newcommand{\be}{\begin{equation}}
\newcommand{\supp}{\mathbf{\mathrm{supp}}}
\newcommand{\ee}{\end{equation}}
\def\Z{{\mathbb Z}}
\def\R{\mathbb{R}}
\def\C{\mathbb{C}}
\let\cal\mathcal
\def\lin{{\rm lin\,}}
\def\cc{}
\let\e\epsilon
\def\cdotG{\rlap{\raisebox{-5pt}{ $\!\scriptstyle G$}}\;{\cdot}\;}

\begin{document}

\title[Quantum semigroups generated by locally compact semigroups] {Quantum semigroups generated by locally compact semigroups}

\author{M.A.\,Aukhadiev, Yu.N.\,Kuznetsova}
\address[M.A.\,Aukhadiev]{Mathematisches Institut, Westf{\"a}lische Wilhelms-Universit{\"a}t M{\"u}nster,  Einsteinstra{\ss}e 62, 48149 M{\"u}nster, Germany; \newline
Kazan State Power Engineering University\\
        Kazan, Russia, 420066 }
        \address[Yu.N.\,Kuznetsova]{Laboratoire de Math\'ematiques\\
        Universit\'e Bourgogne Franche-Comt\'e, 16 route de Gray, 25030 Besan\c con, France}
\email[M.A.\,Aukhadiev]{m.aukhadiev@wwu.de}
\email[Yu.N.\,Kuznetsova]{yulia.kuznetsova@univ-fcomte.fr}

\thanks{The first author is supported in part by the Alexander von Humboldt Foundation, the Russian Foundation for Basic
Research (Grant 14-01-31358) and the Leibniz Fellowship of the Oberwolfach Research Institute for Mathematics (MFO)}

\date{August 25, 2017}

\subjclass{Primary 81R15; 20G42; 16T10; 22A20}
\begin{abstract}
Let $S$ be a subsemigroup of a second countable locally compact group $G$, such that $S^{-1}S=G$. We consider the $C^*$-algebra $C^*_\delta(S)$ generated by the operators of translation by all elements of $S$ in $L^2(S)$. We show that this algebra admits a comultiplication which turns it into a compact quantum semigroup. The same is proved for the von Neumann algebra $VN(S)$ generated by $C^*_\delta(S)$.
\end{abstract}
\maketitle

\section{Introduction}

The notion of a quantum semigroup, as a $C^*$- or von Neumann algebra with a comultiplication, appeared well before the term and before the notion of a locally compact quantum group. But it is especially these last years that substantial examples of quantum semigroups are considered; we would like to mention families of maps on finite quantum spaces \cite{soltan}, quantum semigroups of quantum partial permutations \cite{banica}, quantum weakly almost periodic functionals \cite{daws}, quantum Bohr compactifications \cite{salmi, soltan2}.

In this article we construct a rather ``classical'' family of compact quantum semigroups, which are associated to sub-semigroups of locally compact groups.
The interest of our objects is in fact that they provide natural examples of  $C^*$-bialgebras which are co-commutative and are not however duals of functions algebras. Recall that the classical examples of quantum groups belong to one of the two following types: they are either function algebras, such as the algebra $C_0(G)$ of continuous functions vanishing at infinity on a locally compact group $G$, or their duals, such as the reduced group $C^*$-algebras $C^*_r(G)$. In the semigroup situation one can go beyond this dichotomy.

If $S$ is a discrete semigroup, then the algebra $C^*_\delta(S)$ which we consider coincides with the reduced semigroup $C^*$-algebra $C^*_r(S)$ which has been known since long ago \cite{coburn1,coburn2,barnes, word, paterson}.
If $S=G$ is a locally compact group, then $C^*_\delta(S)=C^*_\delta(G)$ is the $C^*$-algebra generated by all left translation operators in $B(L^2(G))$ \cite{kodaira, bedos}. If $G$ is moreover abelian, then $C^*_\delta(G)$ equals to the algebra $C(\widehat G_d)$ of continuous functions on the dual of the discrete group $G_d$ \cite{kodaira}.

The new case considered in this paper concerns non-discrete nontrivial subsemigroups of locally compact groups, and our objective is to show that their algebras admit a natural coalgebra structure.
Let $G$ be a second countable locally compact group, and let $S$ be its sub-semigroup such that $S^{-1}S=G$. Set $H_S=\{f\in L^2(G): \supp f\subset S\}$; let $E_S$ be the orthogonal projection of $L^2(G)$ onto $H_S$ and let $J_S$ be the right inverse of $E_S$, so that $E_SJ_S = {\rm Id}_{H_S}$.
After the study of semigroup ideals in Section 2, in Section 3 we define $C^*_\delta(S)$ as the $C^*$-algebra generated in $B(H_S)$ by the operators $T_a = E_S L_a J_S$ over all $a\in S$, where $L_a$ is the operator of the left translation by $a$ on $L^2(G)$.

The strong closure of $C^*_\delta(S)$ in $B(H_S)$ is denoted $VN(S)$ and is said to be the semigroup von Neumann algebra. In the case $S=G$, this is the classical group von Neumann algebra, and in the case when the interiour of $S$ is dense in it, this equals to the von Neumann algebra generated by the reduced $C^*$-algebra $C^*_r(S)$ introduced by Muhly and Renault \cite{muhly-renault}, see \hbox{Section 4}.

By defining it first on $VN(S)$, in Section 6 we show that $C^*_\delta(S)$ admits a comultiplication $\Delta$ such that $\Delta(T_a)=T_a\otimes T_a$.
To obtain the main result, we are using techniques of inductive limits and crossed products analogous to the constructions carried out in \cite{Li} and \cite{CEL} for the discrete case. The proof is also based on the duality of semi-lattices in Section \ref{dual}. The discrete abelian case was studied in detail in \cite{AGL}.

The article concludes by a more explicit description of $C^*_\delta(S)$ in the abelian case.

\section{Semigroup ideals}\label{section-ideals}

Let $G$ be a second countable locally compact group, $S$ a closed subsemigroup of $G$ containing the identity $e$ of $G$ and such that $G=S^{-1}S$. Denote by $\mu$ the left Haar measure on $G$. We suppose that $\mu(S)>0$, otherwise our definition would produce a trivially zero algebra; this implies immediately that $S$ has a nonempty interior (apply \cite[20.17]{HR} with $A,B\subset S$ compact of positive measure). One can hold in mind a model example $S=[0,+\infty)$, $G=\R$.

For any subset $X\subset S$ and any $p\in S$, define the {\it translations in $S$}:
\begin{equation}\label{px}
p\cc X=\{pq \colon q\in X\},\ p^{-1}\cc X=\{q\in S\colon pq\in X\}.
\end{equation}
Obviously, $p\cc S$ is a right ideal in $S$, $e\cc S=e^{-1}\cc S=S$ and $p^{-1}\cc S=S$ for any $p\in S$.
For the usual translations in $G$, we use the notation $g\cdotG X = \{gh: h\in X\}$, so that $p^{-1}X =S\cap p^{-1}\cdotG X$.
It is also easy to see that $p\cc (q\cc X)=(pq)\cc X$ and $p^{-1}\cc (q^{-1}\cc X) = (qp)^{-1}\cc X$ for any $X\subset S$ and all $p,q\in S$. We will omit parentheses in the products of this type. Moreover, $p^{-1}p X=X$,
but in general, the products $pq^{-1}$ or $p^{-1}q$ should be viewed purely formally, and $pp^{-1}\cc X$ might differ from $X$ (see, for example, Lemma \ref{s2}).

More precisely,
denote by $\mathcal{F}=\mathcal{F}(S)$ the free monoid generated by $S$ and $S^{-1} \setminus S $. Any element in $\mathcal{F}$ can be canonically reduced, by replacing the expressions $st$, $s^{-1}t^{-1}$ with $s,t\in S$ by $st$ and $s^{-1}t^{-1}$ respectively, to a form of a finite word with alternating symbols in  $S$ and $S^{-1}$. The monoid operation on $\mathcal{F}$ is concatenation of words combined with multiplication in $S$ of neighbouring elements, e.g. $$ (p^{-1}qr^{-1})\cdot(a^{-1}bc^{-1}d)=p^{-1}q(ar)^{-1}bc^{-1}d.$$
The operation of taking inverse in $G$ induces the operation $w\mapsto w^{-1}$ on the monoid $\mathcal{F}$, by $\big(p_1^{\pm1}\cdots p_n^{\pm1}\big)^{-1} = p_n^{\mp1}\cdots p_1^{\mp1}$.

For every $w=p_1^{\pm1}\cdots p_n^{\pm1}\in\mathcal{F}$ and $X\subset S$, define by induction $wX = p_1^{\pm1}(\dots(p_n^{\pm1}X)\dots)$. If $X=S$, then $wS$ is a right ideal in $S$. Define the family of all \emph{constructible right ideals} in $S$ \cite{Li}:
$$\mathcal{J}=\{\bigcap\limits_{i=1}^{n} w_iS\colon w_i\in\mathcal{F} \}\cup\{\emptyset\}.$$

Suppose that $w\in \cal F$ has the form $w=p^{-1}_1q_1p^{-1}_2q_2\dots p^{-1}_nq_n$ with $p_j,q_j\in S$, maybe with $p_1=e$ or $q_n=e$. Then it follows from the definition that $ wS$ is the set of elements $x$ satisfying
\begin{equation}\label{xf} \begin{array}{c}
x=p^{-1}_1q_1\dots p^{-1}_nq_n r_{n+1},\\
 \mbox{ where } r_{n+1}\in S \mbox{ and } \\ r_k=p^{-1}_kq_k\dots p^{-1}_nq_n r_{n+1} \in S \mbox{ for all } k=1,\dots,n.
\end{array}\end{equation}

Define a homomorphism $\mathcal{F}\to G$: $w\mapsto (w)_G$, by $(p^{\pm1})_G=p^{\pm1}$ for $p\in S$.
We fix also an injection  $\imath: G\hookrightarrow \mathcal{F}$ which might not be a homomorphism: for any element $g\in G$ we fix one of its representations $g=p^{-1}q$ and set $\imath(g)=p^{-1}q\in \mathcal{F}$. The notation $gX$, where $g\in G$ and $X\subset S$, is understood in the sense $gX=\imath(g)X$.

\begin{lemma}\label{s1}

For any $w_1,w_2\in\mathcal{F}$, we have $w_1w_2S\subset w_1S$.
\end{lemma}
\begin{proof}
Follows immediately from the facts that $pS\subset S, p^{-1}S= S$ for any $p\in S$.
\end{proof}

\begin{lemma}\label{s2}

For any $w\in\mathcal{F}$, $wS=ww^{-1}S$.
\end{lemma}
\begin{proof}
We can assume that $w$ has the form $w=p^{-1}_1q_1p^{-1}_2q_2\dots p^{-1}_nq_n$ with $p_j,q_j\in S$, maybe with $p_1=e$ or $q_n=e$. Then every $x\in wS$ has the form (\ref{xf}) with $r_{n+1}\in S$ and $r_k=p^{-1}_kq_k\dots p^{-1}_nq_n r_{n+1} \in S$ for all $k=1,\dots,n$.

Now write $x=p^{-1}_1q_1\dots p^{-1}_nq_nq_n^{-1}p_n\dots q_1^{-1}p_1x$. In this product, $x\in S$ and $r_{k+1}=q_k^{-1}p_k\dots q_1^{-1}p_1x\in S$ for $k=1,\dots,n$, as well as $r_k=p^{-1}_kq_k\dots p^{-1}_nq_n q_n^{-1}p_n\dots q_1^{-1}p_1x\in S$ for $k=1,\dots, n$. It follows that $x\in ww^{-1}S$, so $wS\subset ww^{-1}S$. The inverse inclusion follows from Lemma \ref{s1}.
\end{proof}

\begin{lemma}\label{s3}
Let a word $w\in\mathcal{F}$ have the form $w=w_1w_2$, where $w_1,w_2\in\mathcal{F}$. Then $wS=w_1S\cap (w_1)_G\,w_2S$.
\end{lemma}
\begin{proof}
Suppose $w_1=p^{-1}_1q_1p^{-1}_2q_2\dots p^{-1}_iq_i$, $w_2=p^{-1}_{i+1}q_{i+1}p^{-1}_{i+2}q_{i+2}\dots p^{-1}_nq_n$ with $p_j,q_j\in S$. Then every $x\in wS$ satisfies \eqref{xf}. This implies directly that $x\in w_1S$. If we denote $(w_1)_G=p^{-1}q$, $p,q\in S$, then in the notations \eqref{xf} we have also $x=p^{-1}qr_{i+1}$ what implies that $x\in p^{-1}qw_2S = (w_1)_Gw_2S$.

Conversely, if $x\in w_1S\cap (w_1)_Gw_2S$, then

$$x=p^{-1}_1q_1p^{-1}_2q_2\dots p^{-1}_iq_ir_{i+1}',$$
$$r_{i+1}'\in S,\ r_k'= p^{-1}_kq_k\dots p^{-1}_iq_ir_i'\in S\ \mbox{ for }k=1,\dots,i,$$
and $$ x=p^{-1}qp^{-1}_{i+1}q_{i+1}p^{-1}_{i+2}q_{i+2}\dots p^{-1}_nq_n r_{n+1},$$ $$x\in S,\  r_{n+1}\in S,\ r_k= p^{-1}_kq_k\dots p^{-1}_nq_nr_{n+1}\in S\ \mbox{ for }k=i+1,\dots,n.$$
By cancellation, it follows that $r_{i+1}' = p^{-1}_{i+1}q_{i+1}p^{-1}_{i+2}q_{i+2}\dots p^{-1}_nq_n r_{n+1}$, thus in fact the condition \eqref{xf} holds for $x$ and $x\in wS$.
\end{proof}

\begin{corollary}\label{s4}
For any $v,w\in\mathcal{F}$, we have $vS\cap wS=ww^{-1}vS$.
\end{corollary}
\begin{proof}
Since $(ww^{-1})_G=e$, by Lemmas \ref{s2} and \ref{s3} we have
$$wS\cap vS=ww^{-1}S\cap vS=ww^{-1}S\cap (ww^{-1})_GvS=ww^{-1}vS.$$
\end{proof}

It follows directly that $$\mathcal{J}=\{wS|\ w\in F\}\cup \{\emptyset\} .$$

\begin{definition} We will say that measurable subsets $X,Y$ of $G$ are equivalent and write $X\sim Y$ if $\mu(X\Delta Y)=0$, where $\Delta$ denotes the symmetric difference. The equivalence class of $X$ is denoted by $[X]$. \end{definition}

For any $X,X',Y,Y'\in\mathcal{J}$ and $p\in S$ the following holds.

\begin{enumerate}
	\item If $X\sim X'$ and $Y\sim Y'$, then $X\cap Y\sim X'\cap Y'$ and $X\cup Y\sim X'\cup Y'$.
	\item If $X\sim X'$, then $p\cc X\sim p\cc X'$ and $p^{-1}\cc X\sim p^{-1}\cc X'$.
\end{enumerate}
We define, as usual, $[X]\cap[Y]=[X\cap Y]$, $[X]\cup[Y]=[X\cup Y]$, $p\cc[X]=[p\cc X]$, $p^{-1}\cc[X]=[p^{-1}\cc X]$, for any $X,Y\in\mathcal{J}$ and $p\in S$. We will work further with the set  $\mathcal{J}'=\{[X]\colon X\in\mathcal{J}\}$.

The following notion was defined by X.~Li in \cite{Li}. The constructible right ideals of the semigroup $S$ are \emph{independent}, if $X=\cup^n_{j=1}X_j$  for $X,X_1,\dots ,X_n\in\mathcal{J}$, $n\in\mathbb{N}$ implies $X=X_j$ for some $1\leq j\leq n$. This notion is appropriate for a discrete semigroup $S$. In our case this definition should be adjusted.

We say that the constructible right ideals of $S$ are \emph{topologically independent}, if
$$X\sim \cup^n_{j=1}X_j \mbox{ for } X,X_1,\dots ,X_n\in\mathcal{J} \mbox{ implies } $$ $$ X\sim X_j \mbox{ for some } 1\leq j\leq n. $$

Passing to $\mathcal{J}'$, we get the following reformulation of topological independence.
$$[X]= \cup^n_{j=1}[X_j] \mbox{ for } X,X_1,\dots ,X_n\in\mathcal{J} \mbox{ implies } $$ $$[X]=[X_j] \mbox{ for some } 1\leq j\leq n. $$

The two notions in fact do not coincide. We present a simple example of a semigroup, which has not independent constructible right ideals, and at the same time the ideals are topologically independent. 

\begin{example}
	Let $G$ be the group $\R_+\setminus\{0\}$ with respect to the usual multiplication, and consider the subsemigroup $S=\{1 \}\cup\{2\}\cup [3;+\infty)$ in $G$. Computing the constructible right ideals $2S,\ 4^{-1}3S,\ 3^{-1}2S$ we get
	$$2S=\{2\}\cup\{4\}\cup[6;+\infty),$$
	$$4^{-1}3S=[3;+\infty),$$
	$$3^{-1}2S=\{2\}\cup [3;+\infty).$$
	Hence, we have non-independent ideals: $2S\cup 4^{-1}3S= 3^{-1}2S$. At the same time, all the ideals are equivalent to the ideals of the type $[a;+\infty), a\geq 3$, and therefore are topologically independent.
\end{example}

The following is an example of a semigroup whose ideals are not topologically independent.

\begin{example}
	Consider $S=\{0\}\cup[1;1.5]\cup[2;\infty)$ as a subsemigroup of the group $\R$ with respect to usual addition and the usual topology. Further compute the following ideals:
	$$1+S=\{1\}\cup[2;2.5]\cup [3;\infty ),$$
	$$1.5+S=\{1.5\}\cup[2.5;3]\cup [3.5;\infty ),$$
	$$-1.5+(1+S)=\{1\}\cup\{1.5\}\cup [2;\infty ).$$
	We easily see that $-1.5+(1+S)=(1+S)\cup (1.5+S)$, and the same is true for the equivalence classes of these ideals. Hence, the ideals of $S$ are not independent and not topologically independent.  
\end{example}

Both examples above are called \emph{perforated semigroups}, since they are obtained from $\R_+$ by deleting some intervals.

\section{The semigroup C*-algebras}

Further on we will assume that the constructible right ideals of $S$ are topologically independent. It is exactly this property which will guarantee that our comultiplication is well defined.

Consider the Hilbert space $L^2(G)$ with respect to $\mu$. For any measurable subset $X\subset G$ set $H_X=\{f\in L^2(G): {\rm ess}\supp f\subset X\}$; this subspace is isomorphic to $L^2(X,\mu)$. Let $I_X\in L^2(G)$ be the characteristic function of $X$, and let $E_X$ be the orthogonal projection of $L^2(G)$ onto $H_X$, which is just the multiplication by $I_X$. Let $L\colon G\to B(L^2(G)) $ be the left regular representation of $G$, i.e. for any $a,b\in G,\ f\in L^2(G)$
\begin{equation}\label{la} (L_af)(b)=f(a^{-1}b).
\end{equation}
We define the left regular representation $T\colon S\to B(H_S)$ of the semigroup $S$ analogously to $L$. For any $a,b\in S,\ f\in H_S$ we set
\begin{equation}\label{e5}
(T_af)(b)=f(a^{-1}b),
\end{equation}
so that $T_a=E_SL_aE_S$; then
\begin{equation}\label{e6}
(T_a^*f)(b)=I_S(b)f(ab).
\end{equation}

One can verify that $T_a$ is an isometry, $T_a^*T_a=I$, and that for any $f\in H_S$ and $a,b\in S$ we have
$$ (T_aT_a^*f)(b)=I_S(a^{-1}b)f(b).$$
Clearly, $a^{-1}b\in S$ if and only if $b\in aS$, where $aS$ is a constructible right ideal defined in the previous section. Hence the projection   $T_aT_a^*$ is an operator of multiplication by $I_{aS}$. The map $T\colon S\to B(H_S)$ is obviously a representation of $S$.

\begin{definition}
Let $C^*_\delta(S)$ be the $C^*$-subalgebra in $B(H_S)$ generated by the operators $\{T_a: a\in S\}$.
Denote by $VN(S)$ the strong operator closure of $C^*_\delta(S)$ in $B(H_S)$ and call it the \emph{semigroup von Neumann algebra} of $S$.
\end{definition}

If $S=G$, then $C^*_\delta(S)=C^*_\delta(G)$ is the $C^*$-algebra generated by all left translation operators in $B(L^2(G))$ \cite{kodaira, bedos}. If $S$ is discrete, then $C^*_\delta(S)=C^*_r(S)$ is the reduced semigroup $C^*$-algebra \cite{paterson}.

A finite product of the generators $T_a, T_b^*$ for any $a,b\in S$ is called \emph{a monomial}. We will denote also $T_a^*$ by $T_{a^{-1}}$, what does not create confusion in the case $a^{-1}\in S$. Generally, for every $w=p_1^{\pm1}\dots p_n^{\pm1}\in\mathcal F$ we denote $T_w = T_{p_1^{\pm1}}\dots T_{p_n^{\pm1}}$, and clearly every monomial has this form.

\begin{lemma}\label{l5}
For any monomial $T_w$, function $f\in H_S$ and $x\in G$ we have
\begin{equation}\label{e7}
(T_wf)(x)=I_{wS}(x)\cdot f((w^{-1})_Gx)
\end{equation}
\end{lemma}
\begin{proof}
Let $k$ be the length of the word $w$. For $k=1$, either $w=a$ or $w=a^{-1}$ with some $a\in S$. If $w=a$ then for $f\in H_S$ we have
$f(a^{-1}x) = f(a^{-1}x)I_S(a^{-1}x) = I_{aS}(x)f(a^{-1}x)$, thus the expressions \eqref{e5} and \eqref{e7} are equal. If $w=a^{-1}$, then, due to the fact that $a^{-1}S=S$, the formula (\ref{e6}) implies (\ref{e7}).

 Suppose (\ref{e7}) is proved for $k\leq n$ and $w=vw'$ is a word in $\mathcal{F}$ with the length $k+1$, where the length of $v$ and $w'$ is 1 and $k$ respectively. First assume that $v=a\in S$ and denote $g=T_{w'}f$. Then for any $x\in G$ we have
\begin{align*}
(T_wf)(x)&=(T_aT_{w'}f)(x)=(T_ag)(x)
=g(a^{-1}x)
\\&=(T_{w'}f)(a^{-1}x)
=I_{w'S}(a^{-1}x)f((w'^{-1})_Ga^{-1}x)
\\&
=I_{aw'S}(x)f(((aw')^{-1})_Gx)
 = I_{wS}(x)\cdot f((w^{-1})_Gx).
\end{align*}
Now assume that $v=a^{-1}\in S^{-1}$. Then for any $x\in G$ we have

 $$(T_wf)(x)=(T_a^*T_{w'}f)(x)=(T_a^*g)(x)=$$
 $$=I_S(x)g(ax)=I_S(x)(T_{w'}f)(ax)=I_S(x)I_{w'S}(ax)f((w'^{-1})_Gax).$$
 Note that $x\in S$ and $ax\in w'S$ if and only if $x\in a^{-1}w'S$. Thus,
 $$(T_wf)(x)=I_{a^{-1}w'S}(x)f(((a^{-1}w')^{-1})_Gx)=I_{wS}(x)\cdot f((w^{-1})_Gx).$$
The formula (\ref{e7}) follows.
\end{proof}

\begin{lemma}\label{l2}
The $C^*$-algebra $C^*_\delta(S)$ is isomorphic to the $C^*$-subalgebra in $B(L^2(G))$ generated as a \textbf{linear space} by
\begin{equation}\label{EwSLwES}
E_{wS}L_{(w)_G}E_S, \quad w\in\mathcal{F},
\end{equation}
and equivalently by
\begin{equation}\label{EwSLw}
E_{wS}L_{(w)_G}, \quad w\in\mathcal{F}.
\end{equation}
\end{lemma}
\begin{proof}
It follows directly from \eqref{e7} that $T_w E_S = E_{wS} L_{(w)_G} E_S$ for every $w\in\mathcal F$. At the same time, $E_STE_S=TE_S$ for every $T\in C^*_\delta(S)$. Thus, the mapping $T\mapsto TE_S=E_STE_S$ is a *-homomorphism from $C^*_\delta(S)$ to $B(L^2(G))$, and its image is generated exactly by the operators \eqref{EwSLw}. Moreover, this mapping is clearly isometric and thus it is an isomorphism.

To arrive at the second description, one calculates that $L_g E_S = E_{g\cdotG S} L_g$ for every $g\in G$. Thus,
$$
E_{wS}L_{(w)_G}E_S = E_{wS\cap ((w)_G\cdotG S)}L_{(w)_G}.
$$
By definition, $wS\subset (w)_G\cdotG S$, and \eqref{EwSLw} follows.
\end{proof}

The formula (\ref{e7}) shows that $T_w= 0$ if and only if $\mu(wS)=0$, i.e. $wS\sim \emptyset$. For a non-zero monomial $T_w$ define its \emph{index} by $(w)_G\in G$. We have $\mathrm{ind}T_w^*=(w)_G^{-1}$ and $\mathrm{ind}(T_vT_w)=(v)_G(w)_G$, if $T_vT_w\neq 0$. Recall that $E_X\in B(L^2(G))$ is the operator of multiplication by $I_X$.

\begin{corollary}\label{c2} A non-zero monomial $T_w$ in $C^*_\delta(S)$ is an orthogonal projection if and only if $\mathrm{ind}T_w=e$. In this case $T_w=E_{wS}$.

\end{corollary}
\begin{proof}
Let $T_w$ be an orthogonal projection. Then $(ww)_G=(w)_G^2=(w)_G$ and $w_G^{-1}=w_G$. Hence, $(w)_G=\mathrm{ind}T_w=e$.

Suppose that $\mathrm{ind}T_w=e$. Then due to Lemma \ref{l5}, $T_w=E_{wS}$ which is an orthogonal projection.
\end{proof}

\begin{lemma}\label{pro} Every projection $E_X$ for $X\in\mathcal{J}$ is contained in $C^*_\delta(S)$ and equals $T_{ww^{-1}}$ for some $w\in\mathcal{F}$.\end{lemma}
\begin{proof}
By Corollary \ref{s4}, $X=wS$ for some $w\in F$. Due to Corollary \ref{c2}, if $(w)_G=e$ then $E_{wS}\in C^*_\delta(S)$.

Suppose $w$ is an arbitrary element in $\mathcal{F}$. By Lemma \ref{s2}, $wS=ww^{-1}S$ and $E_{wS}=E_{ww^{-1}S} $. Since $(ww^{-1})_G=e$, by Corollary \ref{c2} we have that $E_{wS}=T_{ww^{-1}}\in C^*_\delta(S)$.
\end{proof}

Xin Li \cite[Definition 2.2]{Li} defined the full semigroup $C^*$-algebra of a discrete semigroup as generated not only by isometries associated to the points of $S$, but also by projections corresponding to its constructible ideals. The aim of this construction is to obtain a smaller algebra than just generated by the isometries, so it stays reasonable at least in the case of a commutative semigroup. See \cite{Li} for a longer discussion.
Taking into account the topology on $S$, we adjust the definition of \cite{Li}, using the family $\mathcal{J}'$ instead of $\mathcal{J}$. If $S$ is discrete, the new definition coincides with the old one.  Consider a family of isometries $\{v_p| p\in S\}$ and a family of projections
$\{e_X| X\in\mathcal{J}'\}$ satisfying the following relations for any $p,q\in S$, and $X,Y\in\mathcal{J}'$:
\begin{equation}\label{e3}
 v_{pq}=v_pv_q,\ v_pe_Xv_p^*=e_{pX},
\end{equation}

\begin{equation}\label{e4}
e_S=1,\ e_{\emptyset}=0,\ e_{X\cap Y}=e_Xe_Y.
\end{equation}

\begin{definition}
The universal C*-algebra $C^*(S)$ of the semigroup $S$ is the universal $C^*$-algebra generated by $\{v_p| p\in S\}\bigcup \{e_X| X\in\mathcal{J}'\}$ with the relations above. Since the relations between $T_p$ and $E_X$ in $B(H_S)$ are the same, this algebra is well defined.
Denote by $D(S)$ the commutative $C^*$-algebra generated by the family of projections $\{e_X| X\in\mathcal{J}'\}$ in $C^*(S)$.
\end{definition}



\begin{lemma}\label{left}
There exists a surjective *-homomorphism $\lambda \colon C^*(S)\to C_\delta^*(S)$ such that $\lambda(v_p)=T_p$, $\lambda(e_{[X]})= E_X$. It is extended to a normal *-homomorphism $\lambda: C^*(S)^{**}\to VN(S)$. Both maps will be called the left regular representations, of $C^*(S)$ and $C^*(S)^{**}$ respectively.
\end{lemma}
\begin{proof}
Clearly, $E_X=E_Y$ if and only if $X\sim Y$. Hence, the map $\lambda\colon e_{[X]}\mapsto E_X$ is well defined. Due to the definition of $\cap$ on $\mathcal{J}'$, $\lambda$ is a semigroup isomorphism between the semigroups $\{e_{[X]}\colon [X]\in\mathcal{J}'\}$ and $\{E_X\colon X\in\mathcal{J}\}$. One can easily verify that the operators $T_p$ and $E_X$ satisfy the equations (\ref{e3}) and (\ref{e4}) for all $p\in S,\ X\in\mathcal{J}$. The universality of $C^*(S)$ implies the existence of the homomorphism $\lambda$. Its extension exists by the universality property of $C^*(S)^{**}$.
\end{proof}

Denote by $D_\delta(S)$ the $C^*$-subalgebra in $C^*_\delta(S)$ generated by monomials with index equal to $e$. By Corollary \ref{c2} and Lemma \ref{pro} $D_\delta(S)$ is generated by projections $\{E_X| \ X\in\mathcal{J}\}$, and is obviously commutative.

\begin{lemma}\label{D}
The algebras $D(S)$ and $D_\delta(S)$ are isomorphic.
\end{lemma}
\begin{proof}
By definition, $\lambda(D(S))$ contains the generators of $D_\delta(S)$. Applying Lemma 2.20 in \cite{Li} and using the topological independence of constructible right ideals in $S$ we obtain injectivity of $\lambda|_{D(S)}$.
\end{proof}

There exists a natural action of the semigroup $S$ on the $C^*$-algebra $D_\delta(S)$.
\begin{equation}\label{act}
\tau_p(A)=T_p AT_p^*,\ p\in S,\ A\in D_\delta(S).
\end{equation}
Using the formula (\ref{e7}), we obtain for $A=E_X$, $X\in\mathcal{J}$:
\begin{equation}\label{actex}
\tau_p(E_X)=E_{pX}.
\end{equation}

\section{Comparison with reduced semigroup $C^*$-algebras}

In the case when $S$ is a right reversible left cancellative locally compact semigroup with dense interiour, there exists a construction of the reduced $C^*$-algebra $C^*_{r}(S)$, see \cite{muhly-renault} and a more general construction in \cite{Sundar1}, \cite{Sundar2}. The connection between $C^*_{r}(S)$ and $C^*_\delta(S)$ precisely corresponds to the case of a locally compact group $C^*$-algebras. In particular, as we show further, the von Neumann closures of the two coincide.

We recall that a semigroup is called \emph{right reversible} if every pair of non-empty left ideals has a non-empty intersection (see \cite{Laca}, \cite{Cl}). The following theorem by Ore can be found in \cite{Cl}.

\begin{theorem}\label{Clif}
	A cancellative semigroup $S$ can be embedded into a group $G$ such that $G=S^{-1}S$ if and only if it is right reversible. And in this case, $G$ is a unique up to isomorphism group generated by $S$.
\end{theorem}

This implies in particular, that if $S$ is right reversible with dense interiour, then $C_\delta^*(S)$, as well as $VN(S)$, does not depend on the choice of the group $G$ generated by $S$.

We recall the construction of $C^*_r(S)$ of \cite{muhly-renault,Sundar1} according to the symmetric case $G=S^{-1}S$ adopted here.
Similar to the subspace $L^2(S)\subset L^2(G)$, we can consider $L^1(S)\subset L^1(G)$, which is in addition
 a Banach algebra with respect to convolution. 
For any $f\in L^1(S)$ define an operator $V_f\in B(L^2(S))$:
\begin{equation}\label{vf}
V_f=\int f(a)T_ad\mu(a).\end{equation} 
One can easily verify that $V_f\xi = f*\xi$ for $\xi\in L^2(S)$ and
$$V_f^*=\int \overline{f(a)}T_a^*d\mu(a).$$
The reduced $C^*$-algebra of $S$ is the $C^*$-subalgebra of $B(L^2(S))$ generated by operators $V_f$ over all $f\in L^1(S)$.

\begin{proposition}
	Let $S$ be a locally compact right reversible left cancellative semigroup with dense interiour imbedded in a locally compact group $G$ so that $G=S^{-1}S$. Then $VN(S)$ equals to the von Neumann closure of $C^*_r(S)$ inside $B(L^2(S))$.
\end{proposition}
\begin{proof}
	It is sufficient to show that the commutators of $C^*_r(S)$ and $C^*_\delta(S)$ coincide. If $A\in B(L^2(S))$ commutes with $T_a$ and $T_a^*$ for every $a\in S$, then due to (\ref{vf}) $A$ commutes with every $V_f$ and $V_f^*$ for every $f\in L^1(S)$.
	
	To prove the reverse inclusion, take $\phi_i$ to be the approximate identity of $L^1(G)$ lying in $L^1(S)$ which exists by the assumption that the interiour of $S$ is dense in it. We can assume also that every $\phi_i$ has compact support so that $\phi_i\in L^2(S)$. For any $\xi\in L^2(G)$, $\phi_i*\xi\to\xi$ in the $L^2$-norm. Denote $\phi_{i,a}=T_a\phi_i$, where $a\in S$; one verifies that $V_{\phi_{i,a}}\xi = T_a (\phi_i*\xi)$ for $\xi\in L^2(S)$. Moreover, $T_a\phi_i\in L^1(S)$ so that $V_{\phi_{i,a}}\in C^*_r(S)$.
	
For any $A$ in the commutator of $C^*_r(S)$ and any $\xi\in L^2(S)$,
$$
A T_a\xi = AT_a\lim \phi_i*\xi = \lim AV_{\phi_{i,a}}\xi = \lim T_a (\phi_i*A\xi) = T_a A\xi,
$$	 
so that the two commutators coincide.
	\end{proof}

\begin{example}
There are semigroups to which the construction of \cite{muhly-renault,Sundar1} is not applicable while ours is.
Let $C\subset [0,1]$ be the middle-fifth Cantor set (of positive measure). Let $S$ be the additive semigroup
$$
S = \{0\} \cup \{2\}\cup \cup_{n=0}^\infty (2+2^{-2n-1}+2^{-2n-1}C)\cup [4,+\infty).
$$
Then the interiour of $S$ is not dense in it so $S$ does not satisfy the assumptions of \cite{muhly-renault,Sundar1}. However its ideals are topologically independent, and other assumptions made in Section \ref{section-ideals} are also satisfied.
\end{example}

\section{The semigroup C*-algebra and crossed products}

In what follows we prove a connection between $C_\delta^*(S)$, $VN(S)$ on one hand and the $C^*$- and von Neumann group crossed products by the group $G$ on the other hand.
 The proof is almost a verbatim of Lemma 4.2 of \cite{CEL} and based on Theorem 2.1 of \cite{Laca}. The difference with the mentioned references lies in the topology of the action of $G$ in the outcoming dynamical system.

Define a preorder on $S$: $p\leq q$ if $qp^{-1}\in S$, or equivalently if $q\in Sp$. Due to the assumption $G=S^{-1}S$, we obtain by Theorem \ref{Clif} that for any $p,q\in S$ the left ideals $\{xp\colon x\in S\}$, $\{yq\colon y\in S\}$ have a non-empty intersection. Hence $S$ is upwards directed with respect to this preorder. In the case when $S\cap S^{-1}\ne\{e\}$, this might not be an order; in fact, $p\le q$ and $q\le p$ if $p^{-1}q\in S\cap S^{-1}$.

Consider the directed system of $C^*$-algebras $\mathcal{D}_p$ indexed by $p\in S$, where every $\mathcal{D}_p=D_\delta(S)$. For $p,q\in S$ such that $p\leq q$ we have $qp^{-1}\in S$ and the action (\ref{act}) generates a $*$-homomorphism $\tau_{qp^{-1}}\colon \mathcal{D}_p\to\mathcal{D}_q$:
$$\tau_{qp^{-1}}(A)=T_{qp^{-1}}AT_{qp^{-1}}^*$$
which acts on the generating projections as a translation, see \eqref{actex}.
Clearly, $\tau_{qp^{-1}} = \tau_{qr^{-1}}\tau_{rp^{-1}}$ for $p\le r\le q$. Let $D_\delta^{(\infty)}(S)$ denote the $C^*$-inductive limit of the directed system $\{\mathcal{D}_p,\tau_{qp^{-1}}\}$.

Recall the notation $q^{-1}\cdotG X=\{q^{-1}x\colon x\in X\}\subset G$ for $q\in S$ and $X\in\mathcal{J}$.

\begin{lemma}\label{ind}The $C^*$-algebra $D_\delta^{(\infty)}(S)$ is isomorphic to $$D_G=C^*(\{E_{q^{-1}\cdotG X}\colon q\in S,\ X\in \mathcal{J}\})\subset B(L^2(G)).$$ \end{lemma}

\begin{proof}
By definition, $D_\delta(S)\subset B(H_S)$. Recall that we denote $J_S:H_S\to L^2(G)$ the canonical imbedding; denote by $\pi\colon D_\delta(S)\to B(L^2(G))$ the lifting $\pi(A)= J_S AE_S$.

For any $p\in S$, the map
$$\phi_p(A)=L_p^*\pi(A)L_p,\ \ A\in D_\delta(S),$$
is a $*$-homomorphism $\phi_p\colon D_\delta(S)\to D_G$, such that
$\phi_p(E_X)= E_{p^{-1}\cdotG X}$ for all $X\in\mathcal{J}$.

Then for $p\leq q$ and $X\in\mathcal{J}$ we have:
$$\phi_q \tau_{qp^{-1}}(E_X)=
L_q^*E_{qp^{-1}X}L_q=E_{q^{-1}\cdotG (qp^{-1}X)}.$$
Since $qp^{-1}\in S$, we have $qp^{-1}X = (qp^{-1})\cdotG X$, and we can
continue the previous formula as
$$\phi_q \tau_{qp^{-1}}(E_X)=E_{p^{-1}\cdotG X}=\phi_p(E_X).$$
So the maps $\phi_p$ agree with $\tau_{qp^{-1}}$, and there exists a limit map $\Phi
\colon D^{(\infty)}_\delta(S)\to D_G $, such that $\Phi((A_p)_{p\in S})=\phi_p(A_p)$, $p\in S$. The homomorphisms $\phi_p$ are injective since $\pi$ is obviously injective and $L_p$ is a unitary operator. It follows that $\Phi$ is also injective.

To prove surjectivity of $\Phi$ it suffices to show that for any $q_1,\dots,q_n\in S$, $X_1,\dots, X_n\in\mathcal{J}$ and $\lambda_1,\dots,\lambda_n\in\mathbb{C}$ we have $$\sum\limits_{i}\lambda_i E_{q_i^{-1}\cdotG X_i}\in\Phi(D^{(\infty)}_\delta(S)).$$

Since the system $\{\mathcal{A}_p,\tau_{qp^{-1}}\}$ is upwards directed,
 there exists $s\in S$ such that $q_i\leq s $, $i=1,2,\dots n$, and this implies that $sq_i^{-1}X_i\in\mathcal J$ and $q_i^{-1}\cdotG X_i=s^{-1}\cdotG(sq_i^{-1}X_i)\in \phi_s(D_\delta(S))$. Hence $$\sum\limits_{i}\lambda_i E_{q_i^{-1}\cdotG X_i}\in\phi_s(D_\delta(S))$$  and we obtain
$$D_G=\overline{\bigcup\limits_{p\in S}\phi_p(D_\delta(S))}$$
Therefore $\Phi$ is surjective and we get the isomorphism $D_\delta^{(\infty)}(S)\cong D_G$.\end{proof}

We identify further the $C^*$-algebra $D_\delta^{(\infty)}(S)$ with $D_G$ and in this way consider it to be a subalgebra of $B(L^2(G))$. 

On $B(L^2(G))$, we have the adjoint action of $G$ generated by the left regular representation: $\alpha_g(A)= L_gAL_g^*$, $g\in G$, $A\in B(L^2(G))$. On an operator of multiplication $M_f$ by a function $f\in L^\infty(G)$ it acts by translation: $\alpha_gM_f = M_{L_g f}$, and in particular, $\alpha_g (E_X) = E_{g\cdotG X}$ for $X\in \cal J$.

Let us show that $D_G$ is invariant under this action. For $g\in G$, $q\in S$, $X\in\cal J$
\begin{equation}\label{taug}
\alpha_g (E_{q^{-1}\cdotG X}) = E_{(gq^{-1})\cdotG X}.
\end{equation}
Since $G=S^{-1}S$, we can write $gq^{-1} = t^{-1}s$ with some $s,t\in S$. Then $s\cdotG X = s\cdot X\in \mathcal J$, and $E_{(gq^{-1})\cdotG X} = E_{t^{-1}\cdotG (s X)}\in D_G$.

Moreover, since conjugation is strong operator continuous, $D_G''$ is also invariant under $\alpha$.
One can easily verify that the action $\alpha$ is point-strong continuous on $D_G$. $(D_G'',G,\alpha)$ is thus a von Neumann dynamical system, and by definition the pair  $(Id,L)$ is a covariant representation of this system.

\begin{lemma}\label{crosvn}
The crossed product $D_G''\rtimes_{\alpha}G $ of the commutative von Neumann algebra $D_G''$ and the group $G$ by the action $\alpha$, is isomorphic to the von Neumann algebra $\cal M =\{E_X,L_g\colon X\in S^{-1}\cdot\mathcal{J},g\in G\}'' $.
\end{lemma}
\begin{proof}
By definition (see \cite[Definition X.1.6]{Takesaki}), the crossed product $D_G''\rtimes_{\alpha}G $ is the von Neumann algebra generated by $\{\pi(A): A\in D_G''\}
\cup \{\tilde L_g: g\in G\}\subset B(L^2(G,L^2(G)))\simeq B(L^2(G\times G))$, where
\begin{align*}
(\pi(A)\xi)(s,t) &=\big(\alpha_s^{-1}(A)\xi(s,\cdot)\big)(t),
\\(\tilde L_g \xi)(s,t)&=\xi(g^{-1}s,t)
\end{align*}
for any $\xi\in L^2(G,L^2(G))$, $g,s,t\in G$, $X\in S^{-1}\cdot\mathcal{J}$. Since every $A\in D_G''$ is a multiplication operator $M_f$ with some $f\in L^\infty(G)$, we can write more precisely
$$
(\pi(M_f)\xi)(s,t) =L_{s^{-1}}f(t)\xi(s,t) = f(st)\xi(s,t).
$$
On $L^2(G\times G)$, define a unitary operator $W$:
$$W\xi(s,t)=\delta_G(t)^{\frac{1}{2}}\xi(st,t),$$
which has the adjoint $W^*\xi(s,t) = \delta_G(t)^{-1/2} \xi(st^{-1},t)$.
One verifies that
\begin{align*}
(W(M_f\otimes 1)W^*\xi) (s,t)&=
f(st)\xi(s,t)=(\pi(M_f)\xi)(s,t),
\\W(L_g\otimes 1) W^*&=\tilde L_g,
\end{align*}
so that $D_G''\rtimes_\alpha G = \{ W(A\otimes 1)W^*: A\in D_G''\cup \{L_g:g\in G\}\}''$. It is easy to see that this algebra is isomorphic to $\{ D_G''\cup \{L_g:g\in G\}\}''$, and as a consequence, to $\cal M$.
\end{proof}

\begin{lemma}\label{M-lin-space}
The algebra $\cal M$ is equal to the strong operator closure of the linear space genereated by the operators $E_{q^{-1}\cdotG X} L_g$ with $q\in S$, $X\in\mathcal{J}$, $g\in G$.
\end{lemma}
\begin{proof}
The statement follows from a direct calculation with $q,p\in S$, $X,Y\in\cal J$, $g,h\in G$:
\begin{align*}
E_{q^{-1}\cdot X}L_g E_{p^{-1}\cdot Y} L_h &= E_{q^{-1}\cdot X} (L_gE_{p^{-1}\cdot Y}L_g^*) L_gL_h
\\& = E_{q^{-1}\cdot X} E_{(gp^{-1})\cdot Y} L_{gh} = E_{(q^{-1}\cdot X)\cap((gp^{-1})\cdot Y)} L_{gh}. 
\end{align*}
Represent $gp^{-1}$ as $gp^{-1}=s^{-1}t$, $s,t\in S$. There exists $r\in S$ such that $r\in Sq\cap Ss$; then $rq^{-1},rs^{-1}\in S$ and 
$(q^{-1}\cdot X)\cap((gp^{-1})\cdot Y) = r^{-1}\cdot Z$ with
$$
Z = (rq^{-1}\cdot X)\cap (rs^{-1}t\cdot Y) = (rq^{-1}X)\cap (rs^{-1}tY) \in\cal J.
$$
This shows that the linear space in question is closed under multiplication, what proves the lemma.
\end{proof}

\begin{theorem}\label{crcrossvn}
	The algebra $VN(S)$ is isomorphic to the corner subalgebra $E_S\cal M E_S$ of $\cal M $.
\end{theorem}
\begin{proof}
	
	In fact, the algebra in question is the same as the strong operator closure of the C*-algebra $\cal A$ described in Lemma \ref{l2}, what we will now show.
	By Lemma \ref{M-lin-space}, $E_S\cal M E_S$ is the strong operator closure of the 
	linear space generated by
	 $E_SE_{q^{-1}\cdotG X} L_g E_S$ with $q\in S$, $X\in\mathcal{J}$, $g\in G$.
	Such an operator can be written in another form:
\begin{align*}
		E_SE_{q^{-1}\cdotG X} L_g E_S&=E_{S\cap (q^{-1}\cdot X)} L_g E_S L_g^* L_g
		= E_{q^{-1}X} E_{g\cdot S}L_g =\\
		&=E_{q^{-1}X} E_S E_{g\cdot S} L_g = 		E_{q^{-1}X} E_{S\cap g\cdot S} L_g.
	\end{align*}
		
Let $g=a^{-1}b$ with $a,b\in S$. Then $S\cap g\cdotG S = S\cap (a^{-1}\cdotG (bS)) = a^{-1}bS$.	By Lemma \ref{l2}, $E_{a^{-1}bS}L_{a^{-1}b}\in\cal A$.
Next, if $\Psi: C^*_\delta(S)\to\cal A$ denotes the isomorphism in Lemma \ref{l2}, then by Lemma \ref{pro} $E_{q^{-1}X}=\Psi(T_{ww^{-1}})$ for some $w\in\mathcal{F}$ which depends on $qX$.
	Thus, $E_{q^{-1}X} E_{a^{-1}bS}L_{a^{-1}b}\in \cal A$, so that $E_S\cal ME_S\subset \cal A''$. From the other side, $\cal A$ is generated as a $C^*$-algebra by the operators $E_SL_{a^{-1}b}E_S=\Psi(T_{a^{-1}b})$, $a,b\in S$ which are contained in $E_S\cal M E_S$; this shows that $E_S\cal M E_S=\cal A''$, what proves the theorem.
	
	\end{proof}

\section{The universal and reduced compact quantum semigroups}

\emph{A compact quantum semigroup} is a pair $(A,\Delta)$, where $A$ is a unital C*-algebra and $\Delta\colon A\to A\otimes_{\mathrm{min}} A$ is a unital *-homomorphism which is coassociative, i.e.
$$(\mathrm{id}\otimes \Delta)\Delta=(\Delta\otimes \mathrm{id})\Delta.$$
\emph{A locally compact quantum semigroup} is a pair $(B,\Delta)$, where $B$ is a von Neumann algebra and $\Delta\colon B\to B\overline{\otimes}B$ is a normal unital coassociative *-homomorphism. The homomorphism $\Delta$ in both cases is called \emph{a comultiplication}. See \cite{Timmermann} for details.

Consider the $C^*$-subalgebra $\mathcal{A}$ in $C^*(S)\otimes_{min}C^*(S)$ generated by the elements
$$\{  v_p\otimes v_p, \ e_X\otimes e_X\colon p\in S,\ X\in\mathcal{J}'\}.$$

Clearly, these elements satisfy relations (\ref{e3}), (\ref{e4}). The universal property of $C^*(S)$ implies the existence of a unital $*$-homomorphism $\Delta_u\colon C^*(S)\to\mathcal{A}$, such that
$$\Delta_u(v_p)=v_p\otimes v_p,\ \Delta_u(e_X)=e_X\otimes e_X.$$
The map $\Delta_u$ admits a restriction $\Delta_u|_{D(S)}: D(S)\to D(S)\otimes_{min} D(S)$ which is also a unital *-homomorphism.

The pair $\mathbb{Q}(S)=(C^*(S),\Delta_u)$ is a compact quantum semigroup \cite{AGL}. We call the algebra $C^*(S)$ with this structure the \emph{universal algebra of functions on the compact quantum semigroup} $\mathbb{Q}(S)$ associated with the semigroup $S$.

\subsection{On duality of semilattices and their C*-algebras}\label{dual}

\begin{definition}
	Let $E$ be a semilattice, i.e. a commutative semigroup of idempotents. A \emph{character} on $E$ is a semigroup morphism $\xi\colon E\to \{0,1 \}$. The set of all characters on $E$ is denoted $\hat{E}$; it forms a compact 0-dimensional semilattice with the pointwise multiplication and the topology of pointwise convergence. By Theorem 3.9 Chapter II of \cite{Hofman}, the functor $E\mapsto \hat{E}$ is a duality functor between the category of discrete semilattices and the category of compact 0-dimensional semilattices. In particular, the map $\eta_E\colon E\to \hat{\hat{E}}$ defined by $\eta_E(s)(c)=c(s)$ is an isomorphism. 
	
\end{definition}

\begin{proposition}\label{pl1 }
	Let $E\subset B(H)$ be a set of linearly independent commuting projections closed under multiplication and containing $1_H$. Then $C^*(E)$ is $*$-isomorphic to $C(\hat{E})$, where $\hat{E}$ is the dual semilattice of $E$. Under this isomorphism $\phi$ for every $e\in E$, $\chi\in\hat{E}$, $\phi(e)(\chi)=\chi(e)$.	
\end{proposition}
\begin{proof}
	Note that $E$ is a discrete (in the norm topology) semilattice. Denote $A=C^*(E)$. By Gelfand-Naimark Theorem, the commutative $C^*$-algebra $A$ is $*$-isomorphic to $C(\Omega)$, where $\Omega=\hat{A}$ is the space of characters on $A$ with the topology of pointwise convergence, which is compact and Hausdorff. Obviously, every $\chi\in \hat{A}$ is a character on $E$. Let us show that every $\chi\in \hat{E}$ extends to a continuous character on $A$.
	
	Since $E$ is linearly independent, we can extend $\chi$ by linearity to its linear span $B=\mathrm{lin}E\subset C^*(E)$.
	Let $t=\sum_{j=1}^n \lambda_jp_j\in B$, where $p_j\in E$, $\lambda_j\in \C$. Set $J_0=\{j: \chi(p_j)=0\}$, $J_1=\{j: \chi(p_j)=1\}$. Set $X = \vee \{p_j: j\in J_0\}$ to be the union of projections. Since $X$ is a linear combination of $p_j$, $j\in J_0$ and their products, we have $\chi(X)=0$.
	
	Next, set $Y= \prod_{j\in J_1} p_j$; we have $\chi(Y)=1$. Since $\chi(YX)=0$, we have $Y\ne YX=XY$, what means that $Y(H)\not\subset X(H)$. Pick $v\in Y(H)\cap (X(H))^\perp$ with $\|v\|=1$. Then
	$$
	\|\sum_{j=1}^n \lambda_jp_j\| \ge \|\sum_{j=1}^n \lambda_jp_j v\| = \|\sum_{j\in J_1} \lambda_jp_jv\| = $$
	$$ \|\sum_{j\in J_1} \lambda_jv\| = |\sum_{j\in J_1} \lambda_j| = |\chi(\sum_{j=1}^n \lambda_j p_j)|.
	$$
	This implies that $\chi$ has norm $1$ on the linear span of $E$, so it can be extended to its closure by continuity.
	
	Furthermore, the topologies on $\hat{E}$ and $\hat{A}$ are both defined by pointwise convergence, on $E$ and on $A$ respectively. The bijection defined above is thus continuous in the direction $\hat A\to\hat E$; both spaces being Hausdorff and compact, they are in fact homeomorphic.
\end{proof}

\begin{proposition}\label{pl2 }
	Let $E$ be a set of linearly independent commuting projections on a separable Hilbert space $H$, closed under multiplication and containing $1_H$, and $A=C^*(E)$. Then there exists a positive measure $\mu$ on $\hat{E}$, such that $A''$ is $*$-isomorphic to $L^\infty(\hat{E},\mu)\subset B(L^2(\hat{E},\mu))$.
	
\end{proposition}
\begin{proof}
	By Lemma 4.4.1 in \cite{Murphy}, there exists a separating vector $v\in H$ for the abelian von Neumann algebra $A''$. Analogously to the proof of Theorem 4.4.4 of \cite{Murphy}, projecting $H$ onto its subspace $H_v=[A''v]$ we get that the vector $v$ is cyclic for $A''$ restricted to $H_v$. Since $v$ is separating for $A''$, $A''$ restricted to $H_v$ is $*$-isomorphic to $A''$. 
	
	Denote $\phi\colon A\to C(\Omega)$ the $*$-isomorphism from Proposition \ref{pl1 }, where $\Omega=\hat{E}$. In what follows we reproduce the proof of the Theorem 4.4.3 of \cite{Murphy}.
	Define a positive linear functional $\tau$ on $C(\Omega)$: $\tau(f)=<\phi^{-1}(f)v,v>$. Applying the Riesz-Markov Theorem we obtain a positive measure $\mu$ on $\Omega$ realizing the functional $\tau$. 
	
	Let $\pi$ be the composition of $\phi$ and the $*$-representation of $C(\Omega)$ by multiplication operators on $L^2(\Omega,\mu)$. The map $u\colon Av\to C(\Omega)\subset L^2(\Omega,\mu)$, $av\mapsto \phi(a)$ is linear and isometric, and hence can be extended to a unitary $u$ from $H_v$ onto $L^2(\Omega,\mu)$. In fact, $\pi(a)=uau^{-1}$. Since $\pi(A)$ is strongly dense in $L^\infty(\Omega,\mu)$, we get $uA''u^{-1}=L^\infty(\Omega,\mu)$.	
\end{proof}

\begin{proposition}\label{pl3 }
	Let $E\subset B(H)$ be a set of linearly independent commuting projections on a separable Hilbert space $H$, closed under multiplication and containing $1_H$, and $A=C^*(E)$. Then there exists a comultiplication on $A''$, such that $\Delta(e)=e\otimes e$ for every $e\in E$.
\end{proposition}
\begin{proof}
	By Proposition \ref{pl2 }, $A''$ (resp. $A''\overline{\otimes}A''$) is $*$-isomorphic to $L^\infty(\hat{E},\mu)$ (resp. $L^\infty(\hat{E}\times \hat{E},\mu\times\mu)$). By Theorem 3.9 Chapter II of \cite{Hofman}, the set $\hat{E}$ of characters on $E$ is a compact zero-dimensional semilattice with the (jointly continuous) pointwise product. As in the group case, this product gives rise to a coproduct on $L^\infty(\hat{E},\mu)$ by the formula:
	$$\Delta(f)(x,y)=f(xy)$$
	Since elements of $E$ are characters on $\hat{E}$, we have for every $e\in E$ and $x,y\in \hat{E}$:
	$$\Delta(e)(x,y)=e(xy)=e(x)e(y)= (e\otimes e)(x,y).$$
\end{proof}

\begin{remark}
	The product on $\hat{E}$ is the pointwise product of characters. But this operation is not the pointwise product when the elements of $\hat{E}$ are considered as characters on $C^*(E)$. Namely, for any $\chi_1,\chi_2\in \hat{E}$, $\lambda_i\in\mathbb{C}$, $e_i\in E$, $1\leq i\leq n$ we have:
	$$ \chi_1\chi_2(\sum_{i=1}^{n}\lambda_ie_i)=\sum_{i=1}^{n}\lambda_i\chi_1(e_i)\chi_2(e_i)$$
	
\end{remark}

\subsection{The quantum semigroup associated to $S$}

\begin{theorem}\label{crqvn}
	There exists a comultiplication $\Delta\colon VN(S)\to VN(S)\overline{\otimes}VN(S)$, and its restriction $\Delta\colon C^*_\delta(S)\to C^*_\delta(S)\mathop{\otimes}\limits_{min}C^*_\delta(S)$, with which $\big(VN(S),\Delta\big)$ is a locally compact quantum semigroup and $\mathbb{Q}(S) = \big(C^*_\delta(S),\Delta\big)$ is a compact quantum semigroup.
\end{theorem}
\begin{proof}	
Recall that we suppose $G$ to be second countable, so that $L^2(G)$ is separable.
	We use the Theorem \ref{crcrossvn} to identify $VN(S)$ with the corner $E_S(D_G''\rtimes_{\tau}G)E_S$ inside $B(L^2(G))$. Due to Proposition \ref{pl3 }, there exists a coproduct on $D_G''$, defined on generators by
		$$ \Delta(E_{q^{-1}\cdotG X})= E_{q^{-1}\cdotG X}\otimes E_{q^{-1}\cdotG X}$$
	for $q\in S, X\in\mathcal{J}$.
	One can easily see that $\Delta$ commutes with the action of $G$ on $D_G''$ defined in (\ref{taug}).  Consequently, $\Delta$ gives rise to a comultiplication on $D_G''\rtimes_{\tau}G$, which we also denote by $\Delta$. Due to the fact that $E_S\in D_G''$, using Lemma \ref{crcrossvn} we obtain the required comultiplication $\Delta$ on $VN(S)$.
	
Since $ \Delta(E_{q^{-1}\cdotG X})\in C^*_\delta(S)\otimes C^*_\delta(S)$ for every generator $E_{q^{-1}\cdotG X}$ with $q\in S$, $X\in \cal J$, the map $\Delta$ restricts to a comultiplication on $C^*_\delta(S)$.		\end{proof}


\begin{remark}
The bialgebras $C^*(S)$ and $C^*_\delta(S)$ are co-commutative, as for example the group $C^*$-algebra $C^*(G)$ of $G$. But their dual algebras, unlike the Fourier-Stieltjes algebra $B(G)=C^*(G)^*$, cannot be viewed as function algebras on $S$ or even on $G$.
It is possible that $\phi,\psi\in (C^*_\delta(S))^*$ are nonequal but have the same values on $T_a$ and $T_a^*$ for all $a\in S$.

More specifically, consider for example $G=\Z$, $S=\Z_+$ and $\phi_k(T) = \langle T\delta_k,\delta_k\rangle$, $k\in\Z$. Then $\phi_k(T_a)=\phi_k(T_a^*)=\delta_0(a)$ for all $k\in\Z$, $a\in\Z_+$, but $\phi_k(T_aT_a^*) = I_{\Z_+}(k-a)$ while $\delta_0(T_aT_a^*) = \delta_0(a)$.
\end{remark}

In what follows we use the commutativity assumption just to guarantee that $S$, which is supposed to be right-revrsible, is also left-reversible.

\begin{proposition}\label{ab}
	Let $S$ be abelian. Then $\mu(X)>0$ for any constructible right ideal $X$ of $S$.
\end{proposition}
\begin{proof} Clearly, any ideal of the form $pS$ equals $Sp$ and is not empty. Due to the Theorem \ref{Clif}, the intersection of any pair of non-empty ideals is non-empty. For any non-empty ideal $X$ in $S$ and $p\in S$ we have
	$$p^{-1}X=\{x\in S\colon\ px\in X\}=p^{-1}\cdotG (pS\cap X).$$
	Therefore, $p^{-1}X$ is non-empty for any non-empty ideal $X$ in $S$, and so is $pX$. Hence, every constructible right ideal of $S$ is non-empty.
	
	Let $U$ be any open subset in $S$ with $\mu(U)>0$. Then for any non-empty ideal $X$ in $S$ taking $p\in X$ we obtain $pU\subset X$, hence, $\mu(X)>0$.
	\end{proof}

\begin{remark}\label{comm}
If $S$ is abelian, then there exists a natural short exact sequence connecting $C^*_\delta(S)$ with $C^*_\delta(G)$ described below.

Consider the commutator ideal $K$ in $C^*_\delta(S)$, i.~e.\ the ideal generated by $\{[A,B]=AB-BA\colon A,B\in C^*_\delta(S)\}$.
Among others, $K$ contains the operators
$$
T_aT_a^*-T_a^*T_a=E_{aS}-E_S
$$
for all $a\in S$.
For any $X\in\cal J$, $a\in S$ we have
\begin{align*}
T_a^*(E_X-E_S)T_a&=E_{a^{-1}X}-E_S,\\
T_a(E_X-E_S)T_a^*&=E_{aX}-E_{aS},
\end{align*}
what allows to show by induction that $E_X-E_S\in K$ for every $X\in \cal J\setminus\{\emptyset\}$. Consequently, in $C^*_\delta(S)/K$ we have the equivalence classes $[E_X] = [E_S]=1$ for all $X\in \mathcal{J}\setminus\{\emptyset\}$.

By Lemma \ref{l5}, $T_w = E_{wS} T_{(w)_G}$ for every $w\in\cal F$, and it follows that $[T_w] = [E_S T_{(w)_G}]$ in $C^*_\delta(S)/K$.

Due to Lemma \ref{l2}, $C^*_\delta(S)$ is the closed linear space generated by operators of the form $E_X T_g$ with $X=wS\in\mathcal{J}, g\in G$, $(w)_G=g$ (where $T_g$ is understood as $T_g=T_{a^{-1}b}$ with any representation $g=a^{-1}b$). The discussion above implies that $C^*_\delta(S)/K$ is generated as a linear space by the classes $[E_ST_g]=[T_g]$, $g\in G$.
 
Denote $\hat L_g = [T_g]$. For all $g_1,g_2\in G$ one sees that $\hat L_{g_1} \hat L_{g_2} = \hat L_{g_1g_2}$,
and every $\hat L_g$ is unitary since $\hat L_g^*=\hat L_{g^{-1}}$.

Let us show that $\|\sum_{k=1}^n c_k \hat L_{g_k}\|_{C^*_\delta(S)/K} = \|\sum c_k L_{g_k}\|_{B(L^2(G))}$ for all $c_k\in \mathbb{C}$, $g_k\in G$. We have (by Lemma \ref{l2})
$$
\|\sum_{k=1}^n c_k \hat L_{g_k}\|_{C^*_\delta(S)/K} \le \|\sum c_k T_{g_k}\|_{C^*_\delta(S)}
  \le \|\sum c_k L_{g_k}\|_{B(L^2(G))}.
$$
From the other side, the fact that $S$ has non-empty interior implies that the norm of $T=\sum_{k=1}^n c_k L_{g_k}$ is attained on $H_S$. Indeed, for every $f\in L^2(G)$ and every $\e>0$ there is a compact set $K\subset G$ such that $\|f-E_K f\|<\e$; there exists (see \cite{mukhin}, or alternatively this can be shown directly) $g\in G$ such that $gK\subset S$, so that $E_SE_{gK}=E_{gK}$. Set $h=L_gf$. We have $E_{gK}h = L_g(E_Kf)$, and
\begin{align*}
\|E_Sh-h\| &\le \|E_S(h-E_{gK}h)\|+\|E_SE_{gK}h-h\| 
\\&\le 2\|h-E_{gK}h\| = 2\|f-E_K f\| <2\e.
\end{align*}
If $f$ is chosen so that $\|f\|=1$ and $\|Tf\|>\|T\|-\e$, then $\|h\|=1$,
$$
\|T h\| = \|L_g T f\| = \|Tf\|>\|T\|-\e,
$$
and at the same time $\|T E_Sh-Th\| \le 2\|T\|\e$, what implies $\|TE_Sh\|>\|T\|-(1+2\|T\|)\e$. This proves the statement.

Thus, we have an isomorphism $C^*_\delta(S)/K\simeq C^*_\delta(G)$ and the short exact sequence:
\begin{equation}\label{ses}
0\rightarrow K \rightarrow C^*_\delta(S)\rightarrow C^*_\delta(G)\rightarrow 0.\end{equation}
\end{remark}

\begin{example}
Let us calculate the algebra $C^*_\delta(\R_+)$; we specify that in our notation $\R_+=[0,+\infty)$.
\end{example}
By our assumptions, $G=\R$. It is immediate that
\begin{align*}
\cal J  & = \{ [t,+\infty): t\in\R_+\} \cup\{\emptyset\}.
\end{align*}
Below, denote $E_{[t,+\infty)}=E_t$, $t\in\R$. In the multiplicative notation used throughout the paper, one checks that $a^{-1}bS=[\max(0,b-a),+\infty)$ for $a,b\ge0$, and in general $wS=[t,+\infty)$ with $t\ge \max(0,(w)_\R)$ for every $w\in\cal F$.
Conversely, if $g\in\R$ and $t\ge\max(0,g)$, then $g=t-s$ with $s\ge0$, and $[t,+\infty) = ts^{-1}S$. 
According to Lemma \ref{l2}, we have thus
$$
C^*_\delta(\R_+) \simeq \overline{\lin}\{ E_t L_g: t\in\R_+, g\in\R, t\geq g\}.
$$
It is worth noting that
$$
D(\R_+) = C^*( E_t: t\in\R_+)\subset B(L^2(\R_+)).
$$
This algebra can be also described as the space of functions supported in $\R_+$ and such that $\lim_{t\to t_0-0} f(t)$ exists for every $t_0\in(0,+\infty]$ and $f(t_0)=\lim_{t\to t_0+0} f(t)$ for every $t_0\in[0,+\infty)$. This is the uniform closure of the algebra of piecewise continuous functions, and is sometimes called by the same name.

The  short exact sequence (\ref{ses}) in this case is written as
$$0\rightarrow K \rightarrow C^*_\delta(\R_+)\rightarrow C^*_\delta(\R)\rightarrow 0.$$
The commutator ideal $K$ in $C^*_\delta(\R_+)$ has the following form:
$$
K=\overline{\lin}\{ E_{[a;b)} L_g: a,b\in\R_+, g\in\R, b \geq a\geq g\}.
$$


\end{document}